\documentclass[pdftex]{amsart}	
\usepackage{amsmath, amsfonts, amssymb,amscd,graphics,graphicx,color,eucal,setspace} 
\usepackage{latexsym} 
\usepackage{color}

\usepackage{url}
\usepackage{amscd}
\usepackage{here}
\numberwithin{equation}{section}
\usepackage{tikz}

\theoremstyle{plain}
\newtheorem{theorem}{Theorem}[section]
\newtheorem{lemma}[theorem]{Lemma} 

\newtheorem{proposition}[theorem]{Proposition}

\theoremstyle{definition}
\newtheorem{remark}{Remark}[section]
\newtheorem{example}{Example}[section]
\newtheorem*{definition}{Definition}
\newtheorem*{claim}{Claim}

\def\C{\mathbb C}
\def\R{\mathbb R}
\def\Q{\mathbb Q}
\def\Z{\mathbb Z}

\def\v{\mathbf v}

\def\e{\mathbf e}
\def\0{\mathbf 0}

\def\1{\mathbf 1}
\def\T{\mathbb T}

\DeclareMathOperator{\Hess}{Hess}
\DeclareMathOperator{\Fl}{Fl}
\DeclareMathOperator{\Aut}{Aut}
\DeclareMathOperator{\Hom}{Hom}
\DeclareMathOperator{\GL}{GL}

\begin{document}
\title[Automorphisms of GKM graphs]
{Automorphisms of GKM graphs and regular semisimple Hessenberg varieties}

\author[D. Jang]{Donghoon Jang}
\address{Department of Mathematics, Pusan National University, Pusan, South Korea}
\email{donghoonjang@pusan.ac.kr}

\author[S. Kuroki]{Shintar\^o Kuroki}
\address{Department of Applied Mathematics Faculty of Science, Okayama University of Science, 1-1 Ridai-Cho Kita-Ku Okayama-shi Okayama 700-0005, Okayama, Japan}
\email{kuroki@ous.ac.jp}

\author[M. Masuda]{Mikiya Masuda}
\address{Osaka Central Advanced Mathematical Institute, Osaka Metropolitan University, Sugimoto, Sumiyoshi-ku, Osaka, 558-8585, Japan}
\email{mikiyamsd@gmail.com}

\author[T. Sato]{Takashi Sato}
\address{Osaka Central Advanced Mathematical Institute, Osaka Metropolitan University, Sugimoto, Sumiyoshi-ku, Osaka, 558-8585, Japan}
\email{00tkshst00@gmail.com}
\date{\today}

\author[H. Zeng]{Haozhi Zeng}
\address{School of Mathematics and Statistics, Huazhong University of Science and Technology, Wuhan, 430074, P.R. China}
\email{zenghaozhi@icloud.com}



\begin{abstract}
A regular semisimple Hessenberg variety $\Hess(S,h)$ is a smooth subvariety of the full flag variety $\Fl(\C^n)$ associated with a regular semisimple matrix $S$ of order $n$ and a function $h$ from $\{1,2,\dots,n\}$ to itself satisfying a certain condition.
We show that when $\Hess(S,h)$ is connected and not the entire space $\Fl(\C^n)$,
the reductive part of the identity component $\Aut^0(\Hess(S,h))$ of the automorphism group $\Aut(\Hess(S,h))$ of $\Hess(S,h)$ is an algebraic torus of dimension $n-1$
and $\Aut(\Hess(S,h))/{\Aut^0(\Hess(S,h))}$ is isomorphic to a subgroup of $\mathfrak{S}_n$ or $\mathfrak{S}_n\rtimes \{\pm 1\}$,
where $\mathfrak{S}_n$ is the symmetric group of degree $n$.
As a byproduct of our argument, we show that $\Aut(X)/{\Aut^0(X)}$ is a finite group for any projective GKM manifold $X$.
\end{abstract}

\maketitle

\setcounter{tocdepth}{1}

\section{Introduction}

Let $A$ be an $n\times n$ complex matrix and $h$ a function from $[n]=\{1,\dots, n\}$ to itself satisfying two conditions: 
\[
h(j)\ge j \text{ for any } j\in [n]\quad \text{and}\quad h(1)\le \cdots\le h(n).
\] 
We often denote $h$ as $(h(1),\dots,h(n))$ by listing the values of $h$.
The full flag variety $\Fl(\C^n)$ consists of nested sequences $V_\bullet=(V_1\subset V_2\subset \cdots\subset V_n)$ of linear subspaces of $\C^n$ such that $\dim_\C V_i=i$ for $ i\in [n]$, and a Hessenberg variety $\Hess(A,h)$ is a subvariety of $\Fl(\C^n)$ defined by 
\[
\Hess(A,h)=\{V_\bullet \in \Fl(\C^n)\mid AV_i\subset V_{h(i)} \quad(\forall i\in [n])\}
\]
where $A$ is regarded as a linear operator on $\C^n$.
The family of Hessenberg varieties contains important varieties such as Springer fibers, Peterson varieties, and permutohedral varieties.
See \cite{AH} for recent development on Hessenberg varieties.

Let $S$ be an $n\times n$ complex matrix with distinct eigenvalues.
Such a matrix is called regular semisimple and the corresponding Hessenberg variety $\Hess(S,h)$ is also called regular semisimple.
It is known in \cite{de-pr-sh92} that $\Hess(S,h)$ is a smooth projective variety of complex dimension $\sum_{j=1}^n(h(j)-j)$ and it is connected if and only if $h(j)\ge j+1$ for any $j\in [n-1]$.
Note that $\Hess(S,h)=\Fl(\C^n)$ if and only if $h=(n,\dots,n)$, which immediately follows from the definition of $\Hess(S,h)$.

In this paper, we are concerned with the automorphism group $\Aut(\Hess(S,h))$ of $\Hess(S,h)$ as a variety.
The action of the general linear group $\GL_n(\C)$ on $\C^n$ induces an action on $\Fl(\C^n)$ and the subgroup $\C^*$ consisting of nonzero scalar matrices in $\GL_n(\C)$ acts trivially on $\Fl(\C^n)$.
The induced action of $\mathrm{PGL}_n(\C)=\GL_n(\C)/\C^*$ on $\Fl(\C^n)$ is effective and it is known that $\Aut(\Fl(\C^n))=\mathrm{PGL}_n(\C)\rtimes\{\pm 1\}$ (\cite[Theorems 2 and 3 in pp.75--76]{akhi95}).

The restricted action of the maximal torus $\T$ of $\GL_n(\C)$ commuting with the regular semisimple matrix $S$ leaves $\Hess(S,h)$ invariant and the induced action of the quotient group $\T/\C^*$ on $\Hess(S,h)$ is effective when $\Hess(S,h)$ is connected.
Therefore, we may think of $\T/\C^*$ as a subgroup of the identity component $\Aut^0(\Hess(S,h))$ of $\Aut(\Hess(S,h))$ when $\Hess(S,h)$ is connected.
It follows from a general result on automorphism groups (\cite[Corollary 2.19]{brio18}) that $\Aut^0(\Hess(S,h))$ is a linear algebraic group, so it is a semidirect product of its reductive part and unipotent part.

As is well-known, $\Hess(S,h)$ with this torus action is what is called a GKM manifold.
We denote the GKM graph associated with $\Hess(S,h)$ by $\mathcal{G}_h$.
One can see that the automorphism group $\Aut(\mathcal{G}_h)$ of $\mathcal{G}_h$ is isomorphic to $\mathfrak{S}_n\rtimes \{\pm 1\}$ or $\mathfrak{S}_n$, where $\mathfrak{S}_n$ denotes the symmetric group of degree $n$.

Our main result is the following.

\begin{theorem} \label{theo:main}
Let $\Hess(S,h)$ be connected and not $\Fl(\C^n)$.
Then 
\begin{enumerate}
\item the reductive part of $\Aut^0(\Hess(S,h))$ is $\T/\C^*$, and
\item $\pi_0{\Aut(\Hess(S,h))}=\Aut(\Hess(S,h))/{\Aut^0(\Hess(S,h))}$ is isomorphic to a subgroup of $\Aut(\mathcal{G}_h)$.
\end{enumerate}
\end{theorem}

\begin{remark} \label{rema:1}
(1) When $h=(2,3,\dots,n,n)$, $\Hess(S,h)$ is a toric variety called the permutohedral variety.
In this case, $\Aut^0(\Hess(S,h))=\T/\C^*$ (so the unipotent part of $\Aut^0(\Hess(S,h))$ is trivial) and $\pi_0{\Aut(\Hess(S,h))}$ is isomorphic to the entire group $\Aut(\mathcal{G}_h)$, which follows from Demazure's result on $\Aut(X)$ for a compact smooth toric variety $X$ (\cite[\S 3.4]{oda88}).

(2) The other extreme case where $h=(n-1,n,\dots,n)$ was recently studied by Brosnan et al.~(\cite{br-so24}).
They determine $\Aut(\Hess(S,h))$ in this extreme case.
In particular, they show that when $h=(n-1,n,\dots,n)$ and $n\ge 4$, $\Aut(\Hess(S,h))$ depends on $S$ but $\Aut^0(\Hess(S,h))=\T/\C^*$ regardless of $S$, and that $\pi_0{\Aut(\Hess(S,h))}$ is a proper subgroup of $\Aut(\mathcal{G}_h)$.

(3) When $\Hess(S,h)$ is disconnected, each connected component is isomorphic to each other and a product of smaller regular semisimple Hessenberg varieties.
Since $\Aut^0(X\times Y)\cong \Aut^0(X)\times \Aut^0(Y)$ for any complete varieties $X$ and $Y$ (\cite[Corollary 4.2.7]{br-sa-um13}), one can find the reductive part of $\Aut^0(\Hess(S,h))$ even when $\Hess(S,h)$ is disconnected.
\end{remark}

The idea of the proof of Theorem~\ref{theo:main} is as follows.
Since $\Hess(S,h)$ is a GKM manifold, its (equivariant) cohomology agrees with the (equivariant) graph cohomology of the associated GKM graph $\mathcal{G}_h$ (\cite{go-ko-ma98, gu-za01}).
Let $T$ be the maximal compact subgroup of $\T/\C^*$.
It can be thought of as a subgroup of $\Aut^0(\Hess(S,h))$ when $\Hess(S,h)$ is connected.

Our proof for statement (1) consists of two steps.
The first step is to show that $T$ is maximal among compact tori in $\Aut^0(\Hess(S,h))$.
{We reduce the proof} to the fact that $\Aut^0(\Fl(\C^3))$ is of rank $2$.
The second step is to show that if $G$ is a connected compact Lie subgroup of $\Aut^0(\Hess(S,h))$ containing $T$, then $G=T$.
To be precise, since an element of the normalizer $N_G(T)$ of $T$ in $G$ provides a weakly $T$-equivariant diffeomorphism of $\Hess(S,h)$, it induces an automorphism of the GKM graph $\mathcal{G}_h$, so that we obtain a homomorphism
\begin{equation} \label{eq:Psi}
\Psi\colon N_G(T)\to \Aut(\mathcal{G}_h)
\end{equation}
such that $T\subset \ker\Psi$.
The action of $N_G(T)$ on $H^*(\Hess(S,h))$ is trivial since it is contained in the connected group $\Aut^0(\Hess(S,h))$.
On the other hand, one can determine the group $\Aut(\mathcal{G}_h)$ and it turns out that the action of $\Aut(\mathcal{G}_h)$ on $H^*(\Hess(S,h))$ is essentially Tymoczko's dot action (\cite{tymo08}).
Then, it follows from \cite[Corollary 6.3]{AMS} or \cite[Theorem 6.3]{ki-le24} that the action of a non-identity element of $\Aut(\mathcal{G}_h)$ on $H^*(\Hess(S,h))$ is nontrivial when $\Hess(S,h)\not=\Fl(\C^n)$.
This means that $N_G(T)\subset \ker\Psi$ and it implies $N_G(T)=T$ where we use the compactness of $G$.
Statement (1) in Theorem~\ref{theo:main} follows from this fact.

As for statement (2) in Theorem~\ref{theo:main}, we first prove that $\pi_0{\Aut(\Hess(S,h))}$ is a finite group.
Then, a maximal compact subgroup $G$ of $\Aut(\Hess(S,h))$ with $T$ as the identity component is 
the normalizer of $T$ in $\Aut(\Hess(S,h))$ and $G/T$ is isomorphic to $\pi_0{\Aut(\Hess(S,h))}$.
The homomorphism $\Psi$ in \eqref{eq:Psi} is defined for such $G$ and one can see $\ker\Psi=T$.
This implies statement (2).

For a projective manifold $X$ (i.e.~ smooth projective variety), $\pi_0{\Aut(X)}$ can be an infinite group (see \cite{brio18}).
Our argument for the finiteness of $\pi_0{\Aut(\Hess(S,h))}$ works for any projective GKM manifold and we obtain the following.

\begin{theorem} \label{theo:main2}
If $X$ is a projective GKM manifold, then $\pi_0{\Aut(X)}$ is a finite group.
\end{theorem}

The paper is organized as follows.
In Section~\ref{sect:GKM}, we review the notion of GKM graph, define its automorphism group, and prove properties of the automorphism groups of a GKM manifold and its GKM graph.
Then we prove Theorem~\ref{theo:main2}.
In Section~\ref{sect:Hessenberg}, we determine $\Aut(\mathcal{G}_h)$ and prove (1) and (2) of Theorem~\ref{theo:main} separately.
In Section~\ref{sect:unipotent} we observe that a nontrivial unipotent subgroup of $\mathrm{GL}_n(\C)$ acting on $\Fl(\C^n)$ does not preserve $\Hess(S,h)$ when $\Hess(S,h)\not=\Fl(\C^n)$.
This provides a supporting evidence that the unipotent part of $\Aut^0(\Hess(S,h))$ would be trivial for any $\Hess(S,h)$.

Throughout the paper, all cohomology groups will be taken with integer coefficients unless otherwise stated.

\subsection*{Acknowledgements}
We would like to thank Jaehyun Hong and Eunjeong Lee for comments and informing us of their result \cite{br-so24}.
D.~Jang was supported by the National Research Foundation of Korea(NRF) grant funded by the Korea government(MSIT) (2021R1C1C1004158).
S.~Kuroki was supported by JSPS KAKENHI Grant Number 21K03262.
M.~Masuda was supported in part by JSPS Grant-in-Aid for Scientific Research 22K03292 and the HSE University Basic Research Program.
H.~Zeng was supported by NSFC:11901218.
This work was partly supported by MEXT Promotion of Distinctive Joint Research Center Program JPMXP0723833165.

\section{Automorphisms of GKM graphs and GKM manifolds} \label{sect:GKM}

\subsection{GKM graph and its automorphism}
Let $\Gamma$ be a regular finite graph with vertex set $V(\Gamma)$ and oriented edge set $E(\Gamma)$, where if $e$ is in $E(\Gamma)$ then $e$ with the opposite orientation, denoted by $\bar{e}$, is also in $E(\Gamma)$.
Let $T$ be a compact torus and $BT$ the classifying space of $T$.
Since there is a natural isomorphism between $H^2(BT)$ and the group $\Hom(T,S^1)$ of homomorphisms from $T$ to the unit circle $S^1$ of $\C$, we often think of an element of $H^2(BT)$ as a homomorphism from $T$ to $S^1$ and vice versa.
For $e\in E(\Gamma)$, we denote the initial (resp.\ terminal) vertex of $e$ by $i(e)$ (resp.\ $t(e)$).
An assignment $\alpha\colon E(\Gamma)\to H^2(BT)$ is called an axial function if it satisfies the following three conditions:
\begin{enumerate}
\item $\alpha(\bar{e})=-\alpha(e)$,
\item elements in $\Lambda_p:=\{\alpha(e)\mid i(e)=p\}$ are pairwise linearly independent for any $p\in V(\Gamma)$,
\item $\Lambda_{i(e)}\equiv \Lambda_{t(e)}\pmod{\alpha(e)}$ for any $e\in E(\Gamma)$, that is, there is a bijection between $\Lambda_{i(e)}$ and $\Lambda_{t(e)}$ such that the congruence relation holds for each corresponding pair.
\end{enumerate} 
A GKM graph is a pair $(\Gamma,\alpha)$.
We say that the axial function $\alpha$ is of full rank if the module generated by elements in $\Lambda_p$ is of full rank in $H^2(BT)$ for any $p\in V(\Gamma)$. 

\begin{definition}
An automorphism of a GKM graph $(\Gamma,\alpha)$ is a pair $(\varphi,\phi)$ such that 
\begin{enumerate}
\item $\varphi$ is a graph automorphism of $\Gamma$ and
\item $\phi$ is a group automorphism of $H^2(BT)$ 
\end{enumerate}
satisfying $\phi(\alpha(e))=\alpha(\varphi(e))$ for any $e\in E(\Gamma)$.
The set $\Aut(\Gamma,\alpha)$ of all automorphisms of $(\Gamma,\alpha)$ forms a group under composition of maps.
\end{definition}

\begin{lemma} \label{lemm:Aut_graph_finite}
$\Aut(\Gamma,\alpha)$ is a finite group if $\alpha$ is of full rank.
\end{lemma}
\begin{proof}
Let $V$ be the set of vertices in $\Gamma$.
Then we have a natural homomorphism from $\Aut(\Gamma,\alpha)$ to the permutation group $\mathfrak{S}(V)$ on $V$.
The kernel of this homomorphism is contained in the permutation group $P$ of labels on multiple edges.
Since $\mathfrak{S}(V)$ and $P$ are both finite groups, the lemma follows.
\end{proof}

\subsection{Almost complex $T$-manifold}
Let $M$ be a compact connected almost complex manifold with an action of $T$ preserving the almost complex structure on $M$.
We assume that 
\begin{enumerate}
\item[(A1)] $M^T$ is isolated and weights at the tangential $T$-module $T_pM$ are pairwise linearly independent for any $p\in M^T$.
\end{enumerate}
To such $M$, we associate a GKM graph $(\Gamma_M,\alpha_M)$ as follows.
First the vertex set $V(\Gamma_M)$ is $M^T$.
If $\rho\in H^2(BT)=\Hom(T,S^1)$ is a weight in the complex $T$-module $T_pM$, then the fixed point set component of $\ker\rho$ containing $p$ is a 2-sphere $S_\rho$ by assumption (A1) above.
Therefore, the $T$-action on $S_\rho$ has two fixed points, one is $p$ and the other one, say $q$.
We think of $S_\rho$ as an edge joining $p$ and $q$, and if $e$ is the edge with orientation from $p$ to $q$, then we assign $\rho$ to $e$ and $-\rho$ to $\bar{e}$.
In this way, we obtain the set $E(\Gamma_M)$ of oriented edges and the axial function $\alpha_M$.
An oriented edge $e$ emanating from $p$ may be thought of as an eigenspace of the complex $T$-module $T_pM$, and $\alpha_M(e)$ is the weight of the eigenspace.
{
\begin{remark} \label{rema:finite_group}
When the $T$-action on $M$ is effective, the axial function $\alpha_M$ is of full rank and $\Aut(\Gamma_M,\alpha_M)$ is a finite group by Lemma~\ref{lemm:Aut_graph_finite}.
\end{remark}}  
Let $\Aut(M)$ be the group of all diffeomorphisms of $M$ preserving the almost complex structure on $M$.
By Kobayashi's theorem $\Aut(M)$ is a Lie group (\cite[Corollary 4.2 in p.19]{koba95}).
{We say that $f\in \Aut(M)$ is weakly $T$-equivariant with respect to $\gamma\in \Aut(T)$ if $f(tx)=\gamma(t)f(x)$ for any $t\in T$ and $x\in M$, where $\Aut(T)$ denotes the group of all group automorphisms of $T$.} 

\begin{lemma} \label{lemm:1}
If $f\in \Aut(M)$ is weakly $T$-equivariant with respect to $\gamma\in \Aut(T)$, then $f$ induces an automorphism $(\varphi, \phi)$ of $(\Gamma_M,\alpha_M)$.
Indeed, $\varphi$ on $V(\Gamma_M)=M^T$ is the restriction of $f$ to $M^T$ and $\phi$ is the automorphism $(\gamma^{-1})^*$ of $H^2(BT)=\Hom(T,S^1)$ induced from $\gamma^{-1}$.
\end{lemma} 

\begin{proof}
Since $f$ is a weakly $T$-equivariant diffeomorphism of $M$ preserving the almost complex structure on $M$, the differential of $f$ at $p\in M^T$ sends eigenspaces of $T_pM$ to eigenspaces of $T_{f(p)}M$.
Therefore, $f$ induces a graph automorphism $\varphi$ of $\Gamma_M$.

We shall show that $\phi=(\gamma^{-1})^*$.
Let $p\in M^T=V(\Gamma_M)$ and $v\in T_pM$ a vector in the eigenspace corresponding to an edge $e$ emanating from $p$.
We denote $\alpha_M(e)$ by $\alpha_e$.
Then we have 
\begin{equation} \label{eq:1}
t_* (v)=\alpha_e(t)v\qquad \text{for any $t\in T$}
\end{equation}
where $t_*$ denotes the differential of $t$ (regarded as a diffeomorphism of $M$) at $p$.
Applying the differential $f_*$ of $f$ at $p$ to the both sides of \eqref{eq:1}, we obtain
\begin{equation} \label{eq:2}
f_*(t_*(v))=\alpha_e(t)f_*(v)
\end{equation}
(note that $\alpha_e(t)$ is a complex number).
Here, since $f(tx)=\gamma(t)f(x)$ and $f_*(v)\in T_{f(p)}M$ belongs to the eigenspace corresponding to $\varphi(e)$, we have 
\begin{equation} \label{eq:3}
f_*(t_*(v))=\gamma(t)_*(f_*(v))={\alpha_{\varphi(e)}}(\gamma(t))f_*(v).
\end{equation}
Since $f_*\colon T_pM\to T_{f(p)}M$ is an isomorphism, it follows from \eqref{eq:2} and \eqref{eq:3} that
\[
\alpha_e(t)=\alpha_{\varphi(e)}(\gamma(t)), \quad\text{equivalently}\quad \alpha_e(\gamma^{-1}(t))=\alpha_{\varphi(e)}(t).
\]
Since $\gamma^*(\beta)(t)=\beta(\gamma(t))$ for $\beta\in \Hom(T,S^1)$ by definition, the latter identity above implies 
\[
\alpha_{\varphi(e)}=(\gamma^{-1})^*(\alpha_e), 
\]
proving the desired identity $\phi=(\gamma^{-1})^*$.
\end{proof}

The $T$-action on $M$ can be regarded as a homomorphism 
\begin{equation*} 
T\to \Aut(M).
\end{equation*}
When the $T$-action is effective, the homomorphism above is injective and we think of the torus $T$ as a subgroup of $\Aut(M)$.

\begin{lemma} \label{lemm:2}
Let $f\in \Aut(M)$ be weakly $T$-equivariant.
Suppose that the $T$-action on $M$ is effective and
\begin{enumerate}
\item the automorphism of $(\Gamma_M,\alpha_M)$ induced from $f$ is the identity, 
\item $f$ and $T$ are contained in a compact Lie subgroup of $\Aut(M)$.
\end{enumerate}
Then $f$ is $T$-equivariant, in other words, $f$ commutes with every element of $T$ in $\Aut(M)$.
\end{lemma}

\begin{proof}
Assumption (1) implies that $f$ is the identity on $M^T$ and $f_*\colon T_pM\to T_pM$ $(p\in M^T)$ preserves each eigenspace.
Since each eigenspace is a complex vector space of complex dimension one and $f_*$ is $\C$-linear, $f_*$ must be a scalar multiple on each eigenspace.
Therefore, $f_*$ commutes with the $T$-action on $T_pM$, so the diffeomorphism $h:=(f\circ t)\circ (t\circ f)^{-1}$ of $M$ is the identity on $T_pM$.
This implies that $h$ is the identity on $M$ since $h$ is in a compact Lie group by assumption (2), proving the lemma.
\end{proof}

\begin{remark}\label{rema:h=id}
For the last part in the proof above, we used a general fact that if $G$ is a compact Lie group acting smoothly on a connected manifold $M$ and $h\in G$ fixes a point $p\in M$ and $h_*=\mathrm{id}$ on $T_pM$, then $h=\mathrm{id}$ on $M$.
The argument is as follows.
Since $G$ is a compact Lie group, there is a $G$-invariant Riemannian metric on $M$.
Then the exponential map $\exp\colon T_pM\to M$ is $G$-equivariant.
Since $h_*=\mathrm{id}$ on $T_pM$ and $\exp$ is a $G$-equivariant local diffeomorphism, $h$ fixes a neighborhood of $p$ in $M$.
But, the $h$-fixed point set in $M$ is a closed submanifold (here we again use the assumption that $h\in G$ and $G$ is a compact Lie group) and $M$ is connected, so $h=\mathrm{id}$ on $M$.
\end{remark}

\subsection{Graph cohomology of a GKM graph}
As before, the initial vertex and terminal vertex of $e\in E(\Gamma)$ will be denoted by $i(e)$ and $t(e)$ respectively.
The equivariant graph cohomology of a GKM graph $(\Gamma,\alpha)$ is defined by
\begin{equation} \label{eq:graph_equivariant_cohomology}
H^*_T(\Gamma,\alpha):=\left\{ \xi\in \mathrm{Map}(V(\Gamma),H^*(BT)) \ \left|\ \begin{array}{c}\xi(i(e))\equiv \xi(t(e))\mod {\alpha(e)} \\ \text{ for any } e\in E(\Gamma)\end{array} \right.\right\}.
\end{equation}
We regard elements in $H^*(BT)$ as constant functions in $\mathrm{Map}(V(\Gamma),H^*(BT))$.
Obviously they lie in $H^*_T(\Gamma,\alpha)$ and the (ordinary) graph cohomology of $(\Gamma,\alpha)$ is defined by 
\begin{equation} \label{eq:graph_cohomology}
H^*(\Gamma,\alpha):=H^*_T(\Gamma,\alpha)/(H^{>0}(BT))
\end{equation}
where $(H^{>0}(BT))$ denotes the ideal in $H^*_T(\Gamma,\alpha)$ generated by elements in $H^*(BT)$ with vanishing degree zero term.

\begin{definition} 
We define an action of $(\varphi,\phi)\in \Aut(\Gamma,\alpha)$ on $\xi\in \mathrm{Map}(V(\Gamma),H^*(BT))$ by 
\begin{equation} \label{eq:action}
{((\varphi,\phi)^*\xi)(p):=\phi^{-1}(\xi(\varphi(p))) \quad \text{for $p\in V(\Gamma)$.}}
\end{equation}
Note that this action is contravariant, namely 
\[
((\varphi_1,\phi_1)(\varphi_2,\phi_2))^*=(\varphi_2,\phi_2)^*(\varphi_1,\phi_1)^*.
\]
\end{definition}

\begin{lemma} \label{lemm:3}
The action of $\Aut(\Gamma,\alpha)$ on $\mathrm{Map}(V(\Gamma),H^*(BT))$ preserves $H^*_T(\Gamma,\alpha)$.
\end{lemma}

\begin{proof}
{Let $\xi\in H^*_T(\Gamma,\alpha)$ and $e\in E(\Gamma)$.
Since $\varphi(i(e))=i(\varphi(e))$ and $\varphi(t(e))=t(\varphi(e))$, it follows from \eqref{eq:graph_equivariant_cohomology} that 
\[
\begin{split}
\xi(\varphi(i(e)))- \xi(\varphi(t(e)))&=\xi(i(\varphi(e)))-\xi(t(\varphi(e))) \\
&\equiv 0 \mod{\alpha(\varphi(e))}.
\end{split}
\] 
We apply $\phi^{-1}$ to the both sides above.
Then, noting that $\phi^{-1}(\alpha(\varphi(e)))=\alpha(e)$ (see the definition of $(\varphi,\phi)$ being an automorphism of $(\Gamma,\alpha)$), we obtain 
\[
\phi^{-1}(\xi(\varphi(i(e))))\equiv \phi^{-1}(\xi(\varphi(t(e)))) \mod{\alpha(e)}.
\]
This together with \eqref{eq:action} shows that $(\varphi,\phi)^*\xi$ satisfies the congruence relation in \eqref{eq:graph_equivariant_cohomology}, proving the lemma.} 
\end{proof}

The action of $(\varphi,\phi)$ on $H^*_T(\Gamma,\alpha)$ obviously preserves the ideal $(H^{>0}(BT))$, so it descends to an action on $H^*(\Gamma,\alpha)$.

\subsection{GKM manifold}
In addition to (A1), we further assume that $M$ satisfies
\begin{enumerate}
\item[(A2)] $H^{odd}(M)=0$ (so $H^*(M)$ is torsion free by Poincar\'e duality).
\end{enumerate}
We call such an almost complex $T$-manifold a GKM manifold.
GKM theory (\cite{go-ko-ma98, gu-za01}) tells us that 
\begin{equation} \label{eq:iso_with_Q}
H^*_T(M)\otimes\Q \cong H^*_T(\Gamma_M,\alpha_M) \otimes\Q, \quad H^*(M)\otimes\Q\cong H^*(\Gamma_M,\alpha_M) \otimes\Q
\end{equation}
for a GKM manifold $M$.
To be more precise, since $H^{odd}(M)=0$, the restriction map 
\begin{equation} \label{eq:restriction_iota}
\iota^*\colon H^*_T(M)\to H^*_T(M^T)
\end{equation}
is injective.
Here, since $M^T$ is isolated, we have natural identifications:
\begin{equation} \label{eq:identification}
\begin{split}
H^*_T(M^T)&=H^0(M^T)\otimes H^*(BT)\\
&=\Hom(H_0(M^T),H^*(BT))\\
&=\mathrm{Map}(M^T,H^*(BT)).
\end{split}
\end{equation}
GKM theory says that the image of $\iota^*$ in \eqref{eq:restriction_iota} agrees with $H^*_T(\Gamma_M,\alpha_M)$ when tensored with $\Q$.
\begin{remark}
The isomorphisms in \eqref{eq:iso_with_Q} hold without tensoring with $\Q$ if all the isotropy subgroups for $M$ are connected (\cite{fr-pu06}).
This is the case for the regular semisimple Hessenberg variety $\Hess(S,h)$ with the action of $T$ (or $\widehat{T}$) treated in the next section.
\end{remark}
If $f\in \Aut(M)$ is weakly $T$-equivariant with respect to $\gamma\in \Aut(T)$, then we have an automorphism $f^*$ of $H^*_T(M)$.
We regard $H^*_T(M)$ as a subalgebra of $H^*_T(M^T)$ through $\iota^*$ in \eqref{eq:restriction_iota} and identify $H^*_T(M^T)$ with $\mathrm{Map}(M^T,H^*(BT))$ through \eqref{eq:identification}.
Then one sees that 
\[
(f^*\xi)(p)=\gamma^*(\xi(f(p)))\quad \text{for } \xi\in H^*_T(M)\subset \mathrm{Map}(M^T,H^*(BT)),\  p\in M^T.
\]
Since the automorphism of $(\Gamma_M,\alpha_M)$ induced from $f$ is $(f,(\gamma^{-1})^*)$ by Lemma~\ref{lemm:1}, the identity above together with \eqref{eq:action} and \eqref{eq:iso_with_Q} shows the following.
\begin{lemma} \label{lemm:action_of_Aut}
Let $M$ be a GKM manifold.
If $f\in \Aut(M)$ is weakly $T$-equivariant, then the action of $f$ on $H^*_T(M)\otimes \Q$ and $H^*(M)\otimes \Q$ agrees with the action of the automorphism of $(\Gamma_M,\alpha_M)$ induced from $f$ on them.
\end{lemma}

Motivated by the fact \eqref{eq:iso_with_Q}, we make the following definition.

\begin{definition}
We define $\Aut^*(\Gamma,\alpha)$ to be the subgroup of $\Aut(\Gamma,\alpha)$ acting trivially on $H^*(\Gamma,\alpha) \otimes\Q$.
\end{definition}

\begin{proposition} \label{prop:1}
Let $M$ be a GKM manifold with an effective $T$-action.
If $\Aut^*(\Gamma_M,\alpha_M)$ is trivial, then there is no connected compact non-abelian Lie subgroup of $\Aut(M)$ containing the torus $T$.
\end{proposition}

\begin{proof}
Let $G$ be a connected compact Lie subgroup of $\Aut(M)$ containing $T$ and let $f$ be an element of the normalizer of $T$ in $G$.
Then $f$ is a weakly $T$-equivariant diffeomorphism of $M$ preserving the almost complex structure on $M$, so it induces an automorphism $\hat{f}$ of $(\Gamma_M,\alpha_M)$.
Since $f$ is contained in the connected group $G$, $f$ induces the identity on $H^*(M)\otimes\Q=H^*(\Gamma_M,\alpha_M)\otimes\Q$.
Therefore, $\hat{f}$ is in $\Aut^*(\Gamma_M,\alpha_M)$ and $\hat{f}$ is the identity because $\Aut^*(\Gamma_M,\alpha_M)$ is assumed to be trivial.
Then, it follows from Lemma~\ref{lemm:2} that $f$ commutes with every element of $T$.
This implies that $G$ is abelian, proving the proposition.
\end{proof}

\begin{remark}
Since $\Aut(M)$ is a Lie group, the maximal compact subgroup $K$ of the identity component $\Aut^0(M)$ of $\Aut(M)$ is unique up to conjugation and $\Aut^0(M)$ is homeomorphic to $K\times \R^s$ for some $s$ (in particular $K$ is connected).
Therefore, the proposition above implies that when $M$ is a GKM manifold, the maximal compact subgroup of $\Aut^0(M)$ is a torus if $\Aut^*(\Gamma_M,\alpha_M)$ is trivial.
\end{remark}

We say that an almost complex manifold $M$ has a toral rank $r$ if the rank of $\Aut(M)$, that is the dimension of a maximal compact torus of $\Aut(M)$, is $r$.

\begin{lemma} \label{lemm:K33}
A GKM manifold $M$ which has the bipartite graph $K_{3,3}$ as the underlying graph has toral rank $2$.
\end{lemma} 

\begin{proof}
Since $K_{3,3}$ is a trivalent graph, the real dimension of $M$ is six.
Moreover, since $M$ is a GKM manifold, we have $H^{odd}(M)=0$ and the toral rank of $M$ is at least two.
Suppose that the toral rank of $M$ is three and let $T$ be a 3-dimensional torus in $\Aut^0(M)$.
Then the orbit space $M/T$ is a 3-dimensional manifold with corners and homologically trivial (\cite{ma-pa06}).
Therefore, the boundary of $M/T$ is a 2-sphere.
The underlying graph of the GKM graph of $M$, which is $K_{3,3}$ by assumption, is the 1-skeleton of the manifold with corners $M/T$ and sits in the boundary of $M/T$.
This shows that $K_{3,3}$ is embedded in a $2$-sphere, which contradicts the well-known fact that $K_{3,3}$ is not a planar graph.
\end{proof}

We conclude this section with the proof of Theorem~\ref{theo:main2} in the introduction.

\subsection{Proof of Theorem~\ref{theo:main2}}
Let $X$ be a projective GKM manifold.
Let $T$ be a maximal compact torus in $\Aut(X)$.
We denote by $N(T)$ (resp.\ $N^0(T)$) the normalizer of $T$ in $\Aut(X)$ (resp.\ $\Aut^0(X)$).

\medskip
\noindent
\begin{claim}
The inclusion map $N(T)\hookrightarrow \Aut(X)$ induces an isomorphism 
\[
N(T)/N^0(T)\cong \Aut(X)/{\Aut^0(X)}.
\]
\end{claim}

\begin{proof}
Since $N^0(T)=N(T)\cap \Aut^0(X)$, the injectivity of the induced map above is obvious.
We shall prove the surjectivity.
Take any $[g]\in \Aut(X)/{\Aut^0(X)}$.
Since both $T$ and $gT g^{-1}$ are maximal compact tori in $\Aut^0(X)$, there is an $h\in \Aut^0(X)$ such that $h(gT g^{-1})h^{-1}=T$.
This means that $hg\in N(T)$.
Since $h$ is in $\Aut^0(X)$ and $\Aut^0(X)$ is a normal subgroup of $\Aut(X)$, we have $hg=gh'$ for some $h'\in \Aut^0(X)$.
Therefore, $[hg]=[gh']=[g]$, proving the surjectivity.
\end{proof}

Let $\mathcal{G}_X$ be the GKM graph associated with the projective GKM manifold $X$ with the $T$-action.
Then we have a homomorphism 
\begin{equation} \label{eq:Psi_GKM}
\Psi\colon N(T)\to \Aut(\mathcal{G}_X)
\end{equation}
by Lemma~\ref{lemm:1}.
Since $\Aut(\mathcal{G}_X)$ acts on $H^*(\mathcal{G}_X)$ and $H^*(\mathcal{G}_X)\otimes \Q=H^*(X)\otimes\Q$, we also have a homomorphism 
\[
\pi\colon {\Aut(\mathcal{G}_X)}\to \mathrm{GL}(H^2(\mathcal{G}_X))\to \mathrm{GL}(H^2(X)\otimes\Q),
\] 
where the latter map is the natural one.
Thus we have 
\begin{equation*} 
N(T)\xrightarrow{\Psi}\Aut(\mathcal{G}_X)\xrightarrow{\pi}\mathrm{GL}(H^2(X)\otimes\Q).
\end{equation*}
(Note that $N^0(T)$ is contained in $\ker(\pi\circ\Psi)$ since $N^0(T)$ is contained the connected group $\Aut^0(X)$ so that its acts trivially on $H^2(X)$.) 
As is mentioned in \cite[Section 2]{cant18}, $\Aut^0(X)$ is of finite index in the kernel of the natural map 
\begin{equation*} 
\kappa\colon {\Aut(X)}\to \mathrm{GL}(H^2(X))\hookrightarrow \mathrm{GL}(H^2(X)\otimes\Q), 
\end{equation*}
where the latter map above is injective because $H^2(X)$ is torsion free.
The fact that $\ker\kappa/{\Aut^0(X)}$ is a finite group and Claim above show that 
\begin{equation} \label{eq:finite_index_1}
\ker(\pi\circ \Psi)/N^0(T) \text{ is a finite group}
\end{equation} 
because $\pi\circ\Psi$ agrees with the restriction of the map $\kappa$ to $N(T)$ by Lemma~\ref{lemm:action_of_Aut}.

It follows from \eqref{eq:Psi_GKM} that $N(T)/{\ker\Psi}$ is a subgroup of $\Aut(\mathcal{G}_X)$.
Since $\Aut(\mathcal{G}_X)$ is a finite group by Remark~\ref{rema:finite_group}, so is $N(T)/{\ker\Psi}$.
Therefore 
\begin{equation} \label{eq:finite_index_2}
N(T)/{\ker(\pi\circ\Psi)}\text{ is also a finite group}
\end{equation}
because $\ker(\pi\circ\Psi)$ contains $\ker\Psi$.
It follows from \eqref{eq:finite_index_1} and \eqref{eq:finite_index_2} that $N(T)/N^0(T)$ is a finite group.
This together with Claim above proves Theorem~\ref{theo:main2}.

\section{Regular semisimple Hessenberg varieties} \label{sect:Hessenberg}

A Hessenberg function is a function $h\colon [n]\to [n]$ which satisfies 
\[
h(1)\le h(2)\le \cdots\le h(n)\quad\text{and}\quad h(j)\ge j\quad (\forall j\in [n]).
\]
We often express $h$ as a vector $(h(1), \dots,h(n))$.
Let $S$ be an $n\times n$ matrix with distinct eigenvalues (such a matrix $S$ is called regular semisimple).
Then 
\[
\Hess(S,h)=\{ (V_1\subset V_2\subset \cdots\subset V_n=\C^n)\in \Fl(\C^n)\mid SV_j\subset V_{h(j)}\ (\forall j\in [n])\}
\]
is also called a regular semisimple Hessenberg variety.
$\Hess(S,h)$ has the following properties (\cite{de-pr-sh92}):
\begin{enumerate}
\item It is a smooth projective variety of complex dimension $\sum_{j=1}^n(h(j)-j)$.
\item It is connected if and only if $h(j)\ge j+1$ for any $j\in [n-1]$.
\item $H^{odd}(\Hess(S,h))=0$.
\end{enumerate}

\begin{remark} \label{rema:diff_iso}
The diffeomorphism type of $\Hess(S,h)$ is independent of $S$.
The isomorphism type of $\Hess(S,h)$ as a variety is also independent of $S$ when $h=(2,3,\dots,n,n)$.
In fact, $\Hess(S,h)$ is a toric variety called the permutohedral variety in this case.
However, the isomorphism type of $\Hess(S,h)$ does depend on $S$ when $h=(n-1,n,\dots,n)$ and $n\ge 4$ (\cite{br-so24}).
\end{remark}

\begin{example}[Extreme cases]
When $h_n=(n,\dots,n)$, $\Hess(S,h_n)=\Fl(\C^n)$.
When $h_1=(1,2,\dots,n)$, $\Hess(S,h_1)$ consists of $n!$ points which are permutation flags.
\end{example}

\begin{remark}
If $h(j)=j$ for some $j\in [n-1]$, then $\Hess(S,h)$ is disconnected by (2) above Remark~\ref{rema:diff_iso}.
Their connected components are isomorphic to each other and each one is isomorphic to a product of smaller regular semisimple Hessenberg varieties.
\end{remark}

\begin{lemma} \label{lemm:linear_algebraic_group}
$\Aut^0(\Hess(S,h))$ is a linear algebraic group, so $\Aut^0(\Hess(S,h))$ is a semidirect product of it reductive part and unipotent part.
\end{lemma}

\begin{proof}
Because $H^1(\Hess(S,h))=0$, the Albanese variety of $\Hess(S, h)$ is trivial.
Then the lemma follows from \cite[Corollary 2.19]{brio18}.
\end{proof}

One can easily see from the definition of $\Hess(S,h)$ that $\Hess(S,h)$ is isomorphic to $\Hess(PSP^{-1},h)$ for any $P\in \GL_n(\C)$.
Therefore, we may assume that $S$ is a diagonal matrix with distinct diagonal entries without loss of generality.
Then, the $\C^*$-torus $\T$ consisting of diagonal matrices in $\mathrm{GL}_n(\C)$ commutes with $S$ and the action of $\T$ on $\Fl(\C^n)$ preserves $\Hess(S,h)$.
Since the subgroup $\C^*$ of $\T$ consisting of scalar matrices acts trivially on $\Fl(\C^n)$, the action of $\T$ on $\Hess(S,h)$ induces the action of $\T/\C^*$ on $\Hess(S,h)$.
This induced action is effective when $\Hess(S,h)$ is connected, so we may think of $\T/\C^*$ as a subgroup of $\Aut^0(\Hess(S,h))$ when $\Hess(S,h)$ is connected.

\subsection{GKM graph of $\Hess(S,h)$} \label{subsec:GKM_graph}
Let $T$ be the maximal compact torus of $\T/\C^*$.
The $T$-fixed point set $\Hess(S,h)^T$ consists of permutation flags: 
\[
\e_w:=(\langle \e_{w(1)}\rangle \subset \langle \e_{w(1)},\e_{w(2)}\rangle \subset \cdots \subset \langle\e_{w(1)},\dots,\e_{w(n)}\rangle=\C^n)
\]
where $w$ is an element of the symmetric group $\mathfrak{S}_n$ on $[n]$, $\{\e_1,\dots,\e_n\}$ is the standard basis of $\C^n$, and $\langle \ \ \rangle$ means the linear subspace of $\C^n$ spanned by the elements in $\langle \ \ \rangle$.
We identify the permutation flag $\e_w$ with $w$.

Let $\widehat{T}$ be the maximal compact torus of $\T$, which consists of diagonal matrices $\mathrm{diag}(g_1,\dots,g_n)$ with $g_k\in S^1$ as the $k$-th diagonal entry.
Let ${\pi}_k\colon \widehat{T}\to S^1$ be the projection sending $\mathrm{diag}(g_1,\dots,g_n)$ to $g_k$.
Since $T=\widehat{T}/S^1$ where $S^1$ denotes the group of scalar matrices in $\widehat{T}$,
$\pi_k\pi_\ell^{-1}$ $(1\le k,\ell\le n)$ provides an element of $\Hom(T,S^1)$.
It is known and not difficult to see that the tangential $T$-module $T_w\Hess(S,h)$ at a fixed point $w$ is given by 
\[
T_w\Hess(S,h)=\bigoplus_{j<i\le h(j)}\C({\pi}_{w(i)}{\pi}_{w(j)}^{-1}), 
\] 
where $\C({\pi}_{w(i)}{\pi}_{w(j)}^{-1})$ denotes the complex one-dimensional $T$-module defined by ${\pi}_{w(i)}{\pi}_{w(j)}^{-1}$.

We denote by $t_k$ the element of $H^2(B\widehat{T})$ corresponding to $\pi_k\in \Hom(\widehat{T},S^1)$ through the identification $H^2(B\widehat{T})=\Hom(\widehat{T},S^1)$.
The projection $\widehat{T}\to T$ induces a monomorphism $H^2(BT)\to H^2(B\widehat{T})$ whose image is the submodule 
\begin{equation} \label{eq:submodule}
\left\{\sum_{i=1}^na_it_i \, \left| \, \sum_{i=1}^na_i=0,\ a_i\in \Z \right.\right\}.
\end{equation} 
Henceforth we think of $H^2(BT)$ as the submodule of $H^2(B\widehat{T})$.
The GKM graph associated with $\Hess(S,h)$ as a $T$-manifold, denoted by $(\Gamma_h,\alpha_h)$, is the following (see \cite{tymo08}):
\begin{enumerate}
\item $V(\Gamma_h)=\mathfrak{S}_n$,\quad $E(\Gamma_h)=\{(w,w(i,j))\mid w\in \mathfrak{S}_n,\ j<i\le h(j)\}$,
\item $\alpha_h(w,w(i,j))=t_{w(i)}-t_{w(j)}$ for $w\in \mathfrak{S}_n,\ j<i\le h(j)$,
\end{enumerate}
where $\alpha_h\colon E(\Gamma_h)\to H^2(BT)$ and $(i,j)$ denotes the transposition interchanging $i$ and $j$.
(Note that the GKM graph $(\Gamma_h,\alpha_h)$ is denoted by $\mathcal{G}_h$ in the introduction.)

\subsection{Toral rank of $\Hess(S,h)$}
When $\Hess(S,h)$ is connected, the $T$-action on $\Hess(S,h)$ is effective.
In this subsection we prove the following.

\begin{proposition} \label{prop:toral_rank_Hess}
When $\Hess(S,h)$ is connected, there is no compact torus in $\Aut(\Hess(S,h))$ which properly contains $T$; so the toral rank of $\Hess(S,h)$ is $n-1$.
\end{proposition}

\begin{proof}
Let $\widetilde{T}$ be a compact torus in $\Aut(\Hess(S,h))$ containing $T$ and let $V$ be the linear span of the labels (w.r.t.~$\widetilde{T}$) on the edges (incident to the identity vertex $\mathrm{id}$) corresponding to the simple reflections $(j+1,j)$ for $j\in [n-1]$.
Those labels are $t_{j+1}-t_j$ when restricted to $T$ and the number of the simple reflections are $n-1$, so $\dim V=n-1$.
It suffices to prove the following claim because it implies $\dim\widetilde{T}=n-1$.
We note that the underlying graph of the GKM graph with respect to $\widetilde{T}$ is $\Gamma_h$.

\begin{claim}
The label (w.r.t.~$\widetilde{T}$) on an edge incident to $\mathrm{id}$ is in $V$.
\end{claim}

Proof of the claim.
Let 
\[
H:=\{(i,j)\in [n]\times [n] \mid j < i\le h(j)\}.
\] 
We denote the label on the edge (incident to $\mathrm{id}$) corresponding to $(i,j)\in H$ by $\widetilde{\alpha}_{i,j}$.
We consider a partial order $\prec$ on $H$ defined as follows: 
\begin{quote}
$(p,q)\prec (i,j)$ for $(p,q)\not=(i,j)$ if and only if $p\le i$ and $q\ge j$, 
in other words, $(p,q)$ is located in northeast of $(i,j)$ when elements in $H$ are viewed as positions in an $n\times n$ matrix.
\end{quote}
We shall prove $\widetilde{\alpha}_{i,j}\in V$ by induction on this partial order.
The minimal elements in $H$ with respect to the partial order $\prec$ are $(j+1,j)$ for $j\in [n-1]$ and $\widetilde{\alpha}_{j+1,j}\in V$ for any $j$ by definition of $V$.
Let $(i,j)\in H$ be different from $(j+1,j)$, so $i\ge j+2$.
Suppose that $\widetilde{\alpha}_{p,q}\in V$ for any $(p,q)\in H$ with $(p,q)\prec (i,j)$.
Then $\widetilde{\alpha}_{i-1,j},\ \widetilde{\alpha}_{i,i-1}\in V$ since both $(i-1,j)$ and $(i,i-1)$ are in $H$ and smaller than $(i,j)$ with respect to $\prec$.
The transpositions $(i,j), (i-1,j)$ and $(i,i-1)$ generate a group $S_{i,j}$ isomorphic to $\mathfrak{S}_3$ and the corresponding GKM subgraph $\mathcal{G}_{i,j}$ with vertices $S_{i,j}$ has $K_{3,3}$ as the underlying graph.

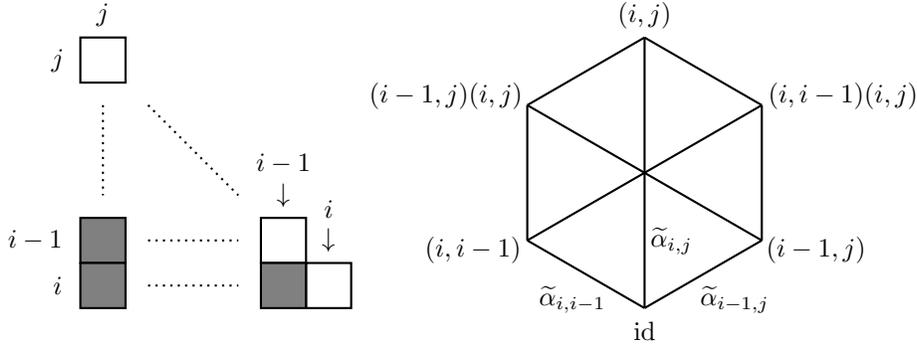
\begin{figure}[H]
\begin{tikzpicture}[scale = 0.6]
\draw[thick, dotted] (2.5,-2.5)--(4.5,-4.5);
\node at (0.5,-1.5) {$j$};
\node at (1.5,-0.5) {$j$};
\node at (0,-5.5) {$i-1$};
\node at (0.5,-6.5) {$i$};
\node at (5.5,-3.8) {$i-1$};
\node at (5.5,-4.5) {$\downarrow$};
\node at (6.5,-4.8) {$i$};
\node at (6.5,-5.5) {$\downarrow$};
\draw [thick] (1,-1) rectangle (2,-2);
\draw [thick, fill = gray] (1,-5) rectangle (2,-6);
\draw [thick, fill = gray] (1,-6) rectangle (2,-7);
\draw [thick] (5,-5) rectangle (6,-6);
\draw [thick, fill = gray] (5,-6) rectangle (6,-7);
\foreach \x in {7}
\draw [thick] (6,-\x+1) rectangle (7,-\x);
\draw[thick, dotted](2.5,-5.5)--(4.5,-5.5);
\draw[thick, dotted](2.5,-6.5)--(4.5,-6.5);
\draw[thick, dotted](1.5,-2.5)--(1.5,-4.5);
\begin{scope}[xshift= 13.5cm, yshift= -4cm]
\draw [thick] (0,3)--(2.598, 1.5);
\draw [thick] (2.598,1.5)--(2.598,-1.5);
\draw [thick] (2.598,-1.5)--(0,-3);
\draw [thick] (0,-3)--(-2.598,-1.5);
\draw [thick] (-2.598,-1.5)--(-2.598,1.5);
\draw [thick] (-2.598,1.5)--(0,3);
\draw [thick] (0,3)--(0,-3);
\draw [thick] (-2.598,1.5)--(2.598,-1.5);
\draw [thick] (-2.598,-1.5)--(2.598,1.5);
\draw(0,-3.5) node{$\mathrm{id}$};
\draw(0,3.5) node{$(i,j)$};
\draw(-3.8,-1.7) node{$(i,i-1)$};
\draw(-4.4,1.7) node{$(i-1,j)(i,j)$};
\draw(3.8,-1.7) node{$(i-1,j)$};
\draw(4.4,1.7) node{$(i,i-1)(i,j)$};
\draw(2,-2.8) node{$\widetilde{\alpha}_{i-1,j}$};
\draw(-1.6,-2.8) node{$\widetilde{\alpha}_{i,i-1}$};
\draw(0.6,-1.5) node{$\widetilde{\alpha}_{i,j}$};
\end{scope}
\end{tikzpicture}
\caption{GKM graph of $Y$ isomorphic to $\Fl(\C^3)$}
\end{figure}

Let $K$ be a codimension two subtorus of $T$ defined by $t_{i-1}=t_i=t_j$.
Then the $K$-fixed point set component $Y$ in $\Hess(S,h)$ containing $\mathrm{id}$ has $\mathcal{G}_{i,j}$ as the GKM graph.
In fact, $Y$ is isomorphic to $\Fl(\C^3)$.
Since $\Aut(\Fl(\C^3))$ is isomorphic to $\mathrm{PGL}_3(\C)\rtimes \{\pm 1\}$, its toral rank is $2$.
Therefore, $\widetilde{\alpha}_{i,j}$ lies in the linear span of $\widetilde{\alpha}_{i-1,j}$ and $\widetilde{\alpha}_{i,i-1}$ and hence in $V$ by induction assumption.
This completes the induction step, proving the claim and hence the proposition.
\end{proof}

\begin{remark} \label{rema:toral_rank}
(1) We may apply Lemma~\ref{lemm:K33} to $Y$ above instead of $\Aut(\Fl(\C^3))$ being of toral rank $2$.

(2) Since the $p$-toral rank of $\Aut(\Fl(\C^3))=\mathrm{PGL}_3(\C)\rtimes \{\pm 1\}$ is $2$ for any prime number $p$, the same argument as above shows that there is no $p$-torus in $\Aut(\Hess(S,H))$ which properly contains the $p$-torus of rank $n-1$ in $T$.
Here, a $p$-torus is an elementary abelian $p$-group and the $p$-toral rank of $\Aut(\Fl(\C^3))$ is the maximal rank of $p$-tori in $\Aut(\Fl(\C^3))$.
\end{remark}

\subsection{Non-abelian part of $\Aut^0(\Hess(S,h))$} 
The purpose of this subsection is to prove the following proposition.

\begin{proposition} \label{prop:abelian}
When $\Hess(S,h)$ is connected and not $\Fl(\C^n)$, there is no connected non-abelian compact Lie subgroup of $\Aut^0(\Hess(S,h))$ containing the torus $T$.
\end{proposition}
We accept this proposition and complete the proof of Theorem~\ref{theo:main}(1).
\begin{proof}[Proof of Theorem~$\ref{theo:main}(1)$]
Propositions~\ref{prop:toral_rank_Hess} and~\ref{prop:abelian} show that the maximal compact Lie subgroup of $\Aut^0(\Hess(S,h))$ is $T$ when $\Hess(S,h)$ is connected and not $\Fl(\C^n)$.
Since $\Aut^0(\Hess(S,h))$ is a linear algebraic group by Lemma~\ref{lemm:linear_algebraic_group}, this implies that the reductive part of $\Aut^0(\Hess(S,h))$ is $\T/\C^*$ in that case, which is Theorem~\ref{theo:main}(1) in the introduction.
\end{proof}

In the following, we show that $\Aut^*(\Gamma_h,\alpha_h)$ is trivial when $\Hess(S,h)$ is connected and not $\Fl(\C^n)$.
Then Proposition~\ref{prop:abelian} follows from Proposition~\ref{prop:1}.
Here are two types of elements in $\Aut(\Gamma_h,\alpha_h)$.
Recall that $H^2(BT)$ is regarded as the submodule \eqref{eq:submodule}, so it is generated by simple roots $t_j-t_{j+1}$ for $j\in [n-1]$.

\begin{example} \label{exam:auto}
(i) For $\sigma\in \mathfrak{S}_n$, we define $(\varphi_\sigma,\phi_\sigma)$ by 
\[
\begin{split}
&\varphi_\sigma(w):=\sigma w\quad \text{for } w\in V(\Gamma_h)=\mathfrak{S}_n,\\
&\phi_\sigma(t_i):=t_{\sigma(i)} \quad \text{for } i\in[n].
\end{split}
\]
Clearly $\varphi_\sigma$ is a graph automorphism of $\Gamma_h$ and 
\[
\begin{split}
\phi_\sigma(\alpha_h(w,w(i,j)))&=\phi_\sigma(t_{w(i)}-t_{w(j)})=t_{\sigma w(i)}-t_{\sigma w(j)}\\
&=\alpha_h(\sigma w,\sigma w(i,j))=\alpha_h(\varphi_\sigma(w, w(i,j))).
\end{split}
\]
Since $\phi_\sigma$ preserves $H^2(BT)$ in $H^2(B\widehat{T})$, 
$(\varphi_\sigma,\phi_\sigma)$ provides an element of $\Aut(\Gamma_h,\alpha_h)$.

\underbar{Note}.
It follows from \eqref{eq:action} that the action of $(\varphi_\sigma,\phi_\sigma)$ on $\xi\in H^*_T(\Gamma_h,\alpha_h)$ is given by
\[
((\varphi_\sigma,\phi_\sigma)^*\xi)(w)
=\phi_\sigma^{-1}(\xi(\varphi_\sigma(w)))=\phi_\sigma^{-1}(\xi(\sigma w))
=\sigma^{-1}(\xi(\sigma w))
\]
where the action of $\sigma$ on $H^*(B\widehat{T})$ is induced from $\sigma(t_i)=t_{\sigma(i)}$.
This shows that the action of $(\varphi_\sigma,\phi_\sigma)$ on $H^*_T(\Gamma_h,\alpha_h)$ is the same as the dot action of $\sigma^{-1}\in \mathfrak{S}_n$, where the dot action (introduced by Tymoczko \cite{tymo08}) is 
\[
(\tau\cdot \xi)(w):=\tau(\xi(\tau^{-1}w)) \qquad \text{for } \tau\in \mathfrak{S}_n.
\]

(ii) Let $w_0$ denote the longest element in $\mathfrak{S}_n$, so $w_0(i) = n+1-i$ for any $i\in [n]$.
Suppose that 
\[
w_0(i)<w_0(j)\le h(w_0(i)) \text{ whenever } j<i\le h(j), \tag{\text{$*$}}
\]
in other words, the configuration on a square grid of size $n\times n$, which consists of boxes in the $i$-th row and the $j$-th column with $i\le h(j)$, remains unchanged when flipped along the anti-diagonal.
Then we define $(\varphi_0,\phi_0)$ by 
\[
\begin{split}
&\varphi_0(w)=w_0ww_0\quad \text{for } w\in V(\Gamma_h)=\mathfrak{S}_n,\\ 
&\phi_0(t_i)=-t_{w_0(i)}(=-t_{n+1-i}).
\end{split}
\] 
Since $w_0^2$ is the identity and $w_0(i,j)w_0=(w_0(i),w_0(j))=(w_0(j),w_0(i))$, we have 
\[
\varphi_0(w,w(i,j))=(\varphi_0(w),\varphi_0(w(i,j)))=(\varphi_0(w),\varphi_0(w)(w_0(j),w_0(i))).
\]
Therefore, the assumption $(*)$ above ensures that $\varphi_0$ preserves edges of $\Gamma_h$, so that $\varphi_0$ is a graph automorphism of $\Gamma_h$.
Moreover, 
\[
\begin{split}
\phi_0(\alpha_h(w,w(i,j)))&=\phi_0(t_{w(i)}-t_{w(j)})=-t_{w_0 w(i)}+t_{w_0 w(j)}\\
&=t_{w_0ww_0(w_0(j))}-t_{w_0ww_0(w_0(i))}\\
&=\alpha_h(\varphi_0(w),\varphi_0(w)(w_0(j),w_0(i)))\\
&=\alpha_h(\varphi_0(w,w(i,j))).
\end{split}
\]
Since $\phi_0$ preserves $H^2(BT)$ in $H^2(B\widehat{T})$, $(\varphi_0,\phi_0)$ provides an element of $\Aut(\Gamma_h,\alpha_h)$ when $(*)$ is satisfied.
Clearly, $(\varphi_0,\phi_0)$ is an involution.

\underbar{Note}.
Let $\Omega$ be the automorphism of $\mathrm{GL}_n(\C)$ defined by
\[
\Omega(A)= w_0(A^T)^{-1}w_0\qquad \text{for } A\in \mathrm{GL}_n(\C),
\]
where $A^T$ denotes the transpose of $A$ and $w_0$ is regarded as the anti-diagonal permutation matrix.
Since $\Omega$ preserves the upper triangular subgroup $B$, it induces an automorphism $\omega$ of $\Fl(\C^n)=\mathrm{GL}_n(\C)/B$, which corresponds to $-1$ in the factor of $\Aut(\Fl(\C^n))=\mathrm{PGL}_n(\C)\rtimes \{\pm 1\}$.
One can see that $\omega$ induces the automorphism $(\varphi_0,\phi_0)$ of $(\Gamma_h,\alpha_h)$ when $\Hess(S,h)=\Fl(\C^n)$.
\end{example}

\begin{lemma} \label{lemm:Aut_Hess}
Set $\Phi_\sigma=(\varphi_\sigma,\phi_\sigma)$ for $\sigma\in \mathfrak{S}_n$ and $\Phi_0=(\varphi_0,\phi_0)$.
Then 
\[
\Aut(\Gamma_h,\alpha_h)=
\begin{cases}
	\langle \Phi_\sigma,\ \Phi_0\mid \sigma\in \mathfrak{S}_n\rangle \quad &\text{if $(*)$ is satisfied,}\\
	\{ \Phi_\sigma\mid \sigma\in \mathfrak{S}_n\} \quad &\text{if $(*)$ is not satisfied.}
\end{cases}
\]
(The former group is isomorphic to $\mathfrak{S}_n\rtimes \langle w_0\rangle$ where $w_0$ acts on $\mathfrak{S}_n$ by conjugation.) 
\end{lemma} 

\begin{proof}
Let $\Phi=(\varphi,\phi)$ be any element of $\Aut(\Gamma_h,\alpha_h)$.
Let $\sigma=\varphi(\mathrm{id})$ where $\mathrm{id}$ is the identity element of $\mathfrak{S}_n$ as before.
Then, since $(\varphi_\sigma^{-1}\circ\varphi)(\mathrm{id})=\mathrm{id}$,
the automorphism $\phi_\sigma^{-1}\circ\phi$ of $H^2(BT)$ permutes the labels on the edges emanating from $\mathrm{id}$,
which are $t_i-t_j$ for $j<i\le h(j)$.
Here, (negative) simple roots $t_{j+1}-t_{j}$ $(j\in [n-1])$ appear in the labels because $h(j)\ge j+1$ for any $j\in [n-1]$ by the connectedness assumption on $\Hess(S,h)$.
Moreover, any root $t_i-t_j$ for $j<i\le h(j)$ is a sum of the (negative) simple roots $t_{j+1}-t_j$ $(j\in [n-1])$.
Therefore, $\phi^{-1}_\sigma\circ\phi$ must permute the (negative) simple roots.

Recall that all the labels on the edges of $(\Gamma_h,\alpha_h)$ are roots in type $A_{n-1}$, i.e.\ $t_k-t_\ell$ $(1\le k\not=\ell\le n)$, and every root in type $A_{n-1}$ appears as a label (see subsection~\ref{subsec:GKM_graph}).
Since $\phi^{-1}_\sigma\circ\phi$ maps roots to roots,
it must be the identity or $\phi_0$ on $H^2(BT)$, where the latter case occurs only when $(*)$ is satisfied.
Thus, the automorphism $\Phi_\sigma^{-1}\circ \Phi$ or $\Phi_0^{-1}\circ\Phi_\sigma^{-1}\circ\Phi$ fixes the vertex $\mathrm{id}$ and induces the identity on $H^2(BT)$.
Therefore, the automorphism must be the identity, proving the lemma.
\end{proof}

For each $i\in [n]$, the element ${\hat x}_i\in \mathrm{Map}(V(\Gamma_h),H^*(B\widehat{T}))$ defined by 
\[
\hat x_i(w):=t_{w(i)}\qquad (w\in V(\Gamma_h)=\mathfrak{S}_n)
\]
provides an element of $H^2_{\widehat{T}}(\Hess(S,h))$ by GKM theory, where the axial function is the same as $\alpha_h$ but with $H^2(B\widehat{T})$ as the target space in place of $H^2(BT)$.
We denote by $x_i$ the image of $\hat{x}_i$ through the natural map $H^2_{\widehat{T}}(\Hess(S,h)) \to H^2(\Hess(S,h))=H^2(\Gamma_h,\alpha_h)$.
Each $x_i$ is fixed under the action of $\Phi_\sigma$ because 
\[
\begin{split}
((\varphi_\sigma,\phi_\sigma)^* {\hat x}_i)(w)
&=\phi_\sigma^{-1}({\hat x}_i(\varphi_\sigma(w)))
=\phi_\sigma^{-1}({\hat x}_i(\sigma w))\\
&=\phi_\sigma^{-1}(t_{\sigma w(i)})=t_{w(i)}={\hat x}_i(w).
\end{split}
\]
\begin{remark}
One can define $x_i$ as the image of some element of $H^2_T(\Hess(S,h))=H^2_T(\Gamma_h,\alpha_h)$, i.e., without using $\widehat{T}$.
However, it is natural to define $x_i$ as the image of $\hat{x}_i$.
Indeed, $\hat{x}_i$ has the following geometrical meaning.
For each $1\le i\le n$, we have a tautological vector bundle $E_i\to \Fl(\C^n)$ of complex dimension $i$ defined by 
\[
E_i:=\{((V_1\subset V_2\subset \cdots\subset V_n), v)\in \Fl(\C^n)\times \C^n\mid v\in V_i\}.
\]
This vector bundle has a natural action of $\widehat{T}$ (but not of $T$).
The $\hat{x}_i$ is the equivariant first Chern class of the complex $\widehat{T}$-line bundle $E_i/E_{i-1}$ restricted to $\Hess(S,h)$, where $E_0=\Fl(\C^n)\times \{0\}$.
As is well-known, when $\Hess(S,h)=\Fl(\C^n)$, 
\[
H^*(\Fl(\C^n))=\Z[x_1,\dots,x_n]/(e_1(x),\dots,e_n(x))
\]
where $e_i(x)$ denotes the $i$-th elementary symmetric polynomial in $x_1,\dots,x_n$.
\end{remark}

Recall that $\Aut^*(\Gamma_h,\alpha_h)$ is the subgroup of $\Aut(\Gamma_h,\alpha_h)$ acting on $H^*(\Gamma_h,\alpha_h)$ trivially.

\begin{lemma} \label{lemm:Aut*Hess}
Suppose that $\Hess(S,h)$ is connected.
If $\Hess(S,h)=\Fl(\C^n)$, then $\Aut^*(\Gamma_h,\alpha_h)=\{\Phi_\sigma\mid \sigma\in \mathfrak{S}_n\}$.
Otherwise $\Aut^*(\Gamma_h,\alpha_h)$ is trivial.
\end{lemma}

\begin{proof}
The sum $\sum_{i=1}^n x_i=e_1(x)$ vanishes in $H^2(\Gamma_h,\alpha_h)$ but there is no other relation among $x_i$'s in $H^2(\Gamma_h,\alpha_h)$, see \cite{AMS} for example.
The action of $\Phi_\sigma$ on $x_i$ is trivial as observed above but $\Phi_0 x_i=-x_{n+1-i}$ because 
\[
\begin{split}
((\varphi_0,\phi_0)^*\hat{x}_i)(w)&=\phi_0^{-1}({\hat x}_i(\varphi_0(w)))=\phi_0^{-1}({\hat x}_i(w_0ww_0))=\phi_0^{-1}(t_{w_0ww_0(i)})\\
&=\phi_0^{-1}(t_{n+1-w(n+1-i)})=-t_{w(n+1-i)}=-{\hat x}_{n+1-i}(w).
\end{split}
\] 
Since $H^*(\Fl(\C^n))$ is generated by $x_i$'s as a ring, this together with Lemma~\ref{lemm:Aut_Hess} proves the former statement in the lemma.

When $\Hess(S,h)\not=\Fl(\C^n)$, $H^*(\Gamma_h,\alpha_h)$ under the dot action of $\mathfrak{S}_n$ contains an $\mathfrak{S}_n$-module isomorphic to the standard $\mathfrak{S}_n$-module (\cite[Corollary 6.3]{AMS} or \cite[Theorem 6.3]{ki-le24}),
which means that $\Phi_\sigma$ is not in $\Aut^*(\Gamma_h,\alpha_h)$ unless $\sigma=\mathrm{id}$.
When $(*)$ is satisfied, $\Phi_\sigma\circ \Phi_0$ is in $\Aut(\Gamma_h,\alpha_h)$ by Lemma~\ref{lemm:Aut_Hess} but since 
\[
(\Phi_\sigma\circ\Phi_0)(x_i)=\Phi_\sigma(-x_{n+1-i})=-x_{n+1-i},
\]
$\Phi_\sigma\circ\Phi_0$ is not in $\Aut^*(\Gamma_h,\alpha_h)$.
This together with Lemma~\ref{lemm:Aut_Hess} proves the latter statement in the lemma.
\end{proof}

Proposition~\ref{prop:abelian} follows from Proposition~\ref{prop:1} and Lemma~\ref{lemm:Aut*Hess}.

\subsection{Proof of Theorem~\ref{theo:main}(2)}
Since $\pi_0{\Aut(\Hess(S,h))}$ is a finite group by Theorem~\ref{theo:main2},
there exists a maximal compact subgroup $G$ of $\Aut(\Hess(S,h))$ such that $\Aut(\Hess(S,h))/G$ is homeomorphic to a Euclidean space and that maximal compact subgroups are conjugate to each other by inner automorphisms (\cite[Chapter VII]{bore98}).
Therefore, $\pi_0{\Aut(\Hess(S,h))}$ is isomorphic to $G/G^0$, where $G^0$ denotes the identity component, and we may think of $G^0$ as $T$ by Proposition~\ref{prop:abelian}.
Since the normalizer of $G^0$ in $G$ is the entire $G$, any element $g$ of $G$ induces an automorphism of the GKM graph $(\Gamma_h,\alpha_h)$, so that we obtain a homomorphism 
\[
\Psi\colon G\to \Aut(\Gamma_h,\alpha_h).
\]
We know that $\ker\Psi$ contains $G^0$ and Lemma~\ref{lemm:2} shows that $\ker\Psi$ is abelian.
Therefore, if $\ker\Psi$ properly contains $G^0$, then $\ker\Psi$ is isomorphic to $G^0\times K$ for some non-trivial finite abelian group $K$.
Therefore, there exists a $p$-torus in $\ker\Psi$ which properly contains the $p$-torus of rank $n-1$ in $G^0=T$ for some prime number $p$.
This contradicts Remark~\ref{rema:toral_rank}(2), so $\ker\Psi=G^0$ and hence $\Psi$ induces a monomorphism
\[
G/G^0\hookrightarrow \Aut(\Gamma_h,\alpha_h).
\] 
Since $G/G^0$ is isomorphic to $\pi_0{\Aut(\Hess(S,h))}$ as remarked above, this proves Theorem~\ref{theo:main}(2).

\section{An observation} \label{sect:unipotent}

The general linear group $\GL_n(\C)$ naturally acts on $\Fl(\C^n)$.
Let $I$ be the identity matrix in $\GL_n(\C)$ and $E_{ij}$ $(1\le i\not=j\le n)$ the matrix of order $n$ with $1$ in the $(i,j)$ entry and $0$ otherwise.
We shall observe that the restricted action of a unipotent subgroup 
\[
U_{ij}:=\{I+cE_{ij}\mid c\in \C\}\subset \GL_n(\C)
\]
does not preserve $\Hess(S,h)$ when $\Hess(S,h)\not=\Fl(\C^n)$, where $S$ is a diagonal matrix with distinct diagonal entries as before.

Let $H$ be the vector space consisting of square matrices $(b_{ij})$ of order $n$ such that $b_{ij}=0$ with $i>h(j)$.
Through the identification $\Fl(\C^n)=\GL_n(\C)/B$, where $B$ denotes the Borel subgroup of $\GL_n(\C)$ consisting of upper triangular matrices as before, one has 
\[
\Hess(S,h)=\{gB\in \GL_n(\C)/B \mid g^{-1}Sg\in H\}.
\] 
Indeed, to $g=[\v_1,\v_2,\dots,\v_n]\in \GL_n(\C)$ we associate a complete flag in $\C^n$:
\[
\langle \v_1\rangle\subset \langle \v_1,\v_2\rangle \subset \cdots \subset \langle \v_1,\v_2,\dots,\v_n\rangle=\C^n
\]
where $\langle \ \ \rangle$ denotes the vector space spanned by the vectors therein.
This induces the identification $\GL_n(\C)/B=\Fl(\C^n)$ and the condition $SV_i\subset V_{h(i)}$ $(\forall i\in [n])$ is equivalent to $Sg\in gH$, i.e.~ $g^{-1}Sg\in H$.

Suppose that the restricted action of $U_{ij}$ on $\Fl(\C^n)$ preserves $\Hess(S,h)$.
Then 
\[
\begin{split}
H\ni ((I+cE_{ij})g)^{-1}S(I+cE_{ij})g&=g^{-1}(I-cE_{ij})S(I+cE_{ij})g\\
&=g^{-1}(S+cSE_{ij}-cE_{ij}S)g\\
&=g^{-1}Sg+cg^{-1}(SE_{ij}-E_{ij}S)g\\
&=g^{-1}Sg+cg^{-1}(d_i-d_j)E_{ij}g
\end{split}
\]
for any $gB\in \Hess(S,h)$ and $c\in \C$, where $d_k$ for $k=i,j$ is the $k$-th diagonal entry of $S$.
Since $g^{-1}Sg\in H$ and $d_k$'s are mutually distinct, the above computation shows that 
\begin{equation} \label{eq:Eij}
g^{-1}E_{ij}g\in H.
\end{equation}
If $g=(a_{k\ell})\in \mathrm{SL}_n(\C)$, then an elementary computation shows that the $(n,1)$ entry of $g^{-1}E_{ij}g$ is given by $\widetilde{a}_{in}a_{j1}$ where $\widetilde{a}_{in}$ denotes the cofactor of $a_{in}$.
Remember that any permutation flag is in $\Hess(S,h)$.
Therefore, \eqref{eq:Eij} must be satisfied for any permutation matrix $g$.
As $g$, we take a permutation matrix such that $(j,1)$ entry and $(i,n)$ entry are both $1$.
Then the $(n,1)$ entry of $g^{-1}E_{ij}g$ is $\pm 1$.
However, if $\Hess(S,h)\not=\Fl(n)$, that is $h\not=(n,\dots,n)$, then $h(1)\le n-1$ so that $b_{n1}=0$ for any $(b_{ij})\in H$.
This is a contradiction.
Thus, the restricted action of $U_{ij}$ on $\Fl(\C^n)$ does not preserve $\Hess(S,h)$ when $\Hess(S,h)\not=\Fl(\C^n)$.


\begin{thebibliography}{9}

\bibitem{AH}
H. Abe and T. Horiguchi,
\textit{A survey of recent developments on Hessenberg varieties}, 
Springer Proc. Math. Stat., 332, Springer, Singapore, 2020, 251--279.

\bibitem{akhi95}
D. N. Akhiezer, 
\textit{Lie Group Actions in Complex Analysis}, 
Vieweg$+$Teubner Verlag; 1995th edition.

\bibitem{AMS}
A. Ayzenberg, M. Masuda, and T. Sato, 
\textit{The second cohomology of regular semisimple Hessenberg varieties from GKM theory}, 
Proceedings of the Steklov Institute of Mathematics, 2022, Vol. 317, 1--20. 

\bibitem{bore98}
A. Borel, 
\textit{Semisimple groups and Riemannian symmetric spaces}, 
Texts Read. Math., 16 Hindustan Book Agency, New Delhi, 1998.

\bibitem{brio18}
M. Brion, 
\textit{Notes on automorphism groups of projective varieties}, 
\url{https://www-fourier.univ-grenoble-alpes.fr/~mbrion/autos_final.pdf}

\bibitem{br-sa-um13}
M. Brion, P. Samuel, and V. Uma,
\textit{Lectures on the structure of algebraic groups and geometric applications}, 
CMI Lect. Ser. Math., 1
Hindustan Book Agency, New Delhi; Chennai Mathematical Institute (CMI), Chennai, 2013.

\bibitem{br-so24}
P. Brosnan, L. Escobar, J. Hong, D. Lee, E. Lee, A. Mellit, and E. Sommers, 
\textit{Automorphisms and deformations of regular semisimple Hessenberg varieties}, in preparation.

\bibitem{cant18}
S. Cantat, 
\textit{Automorphisms and dynamics: a list of open problems}, Proc. Int. Cong. of Math. 2018, 
Rio de Janeiro, Vol. 2, 637--652. 

\bibitem{de-pr-sh92}
F. De Mari, C. Procesi, and M. A. Shayman,
\textit{Hessenberg varieties}, 
Trans. Amer. Math. Soc. {332} (1992), no. 2, 529--534. 


\bibitem{fr-pu06}
M. Franz and V. Puppe,
\textit{Exact cohomology sequences with integral coefficients for torus actions}, 
Transform. Groups 12 (2006), 65--76.localization


\bibitem{go-ko-ma98}
M. Goresky, R. Kottwitz and R. MacPherson,
\textit{Equivariant cohomology, Koszul duality and the localization theorem},
Invent. Math.~131 (1998), 25--83.

\bibitem{gu-za01}
V. W. Guillemin and C. Zara,
\textit{One-skeleta, Betti numbers and equivariant cohomology},
Duke Math. J.~107 (2001), 283--349.

\bibitem{ki-le24}
Y-H. Kiem and D. Lee,
\textit{Birational geometry of generalized Hessenberg varieties and the generalized Shareshian-Wachs conjecture},
J. Combin. Theory Ser. A, Vol. 206, August 2024, 105884,
\url{https://www.sciencedirect.com/science/article/abs/pii/S0097316524000232}.

\bibitem{koba95}
S. Kobayashi, 
\textit{Transformation Groups in Differential Geometry},
Springer; 1995th edition.

\bibitem{ma-pa06}
M. Masuda and T. Panov,
\textit{On the cohomology of torus manifolds}, 
Osaka J. Math. 43 (2006), 711--746. 

\bibitem{oda88}
T. Oda,
\textit{Convex bodies and algebraic geometry} (An Introduction to the 
theory of toric varieties), Ergebnisse der Mathematik und ihrer Grenzgebiete, 
3. Folge Band 15, A Series of Modern Surveys in Mathematics, Springer-Verlag, 1988. 

\bibitem{tymo08}
J. S. Tymoczko,
\textit{Permutation actions on equivariant cohomology of flag varieties},
Toric Topology, Contemp. Math. 460 (2008), 365--384.

\end{thebibliography}
\end{document}